\documentclass[reqno]{article}
\usepackage{authblk}
\pdfoutput=1
\usepackage{amsmath}
\usepackage{amsfonts,amsthm,amssymb}
\usepackage{mathabx}
\usepackage{bm}
\usepackage{amscd}
\usepackage[latin2]{inputenc}
\usepackage{t1enc}
\usepackage[mathscr]{eucal}
\usepackage{indentfirst}
\usepackage{graphicx}
\usepackage{graphics}
\usepackage{pict2e}
\usepackage{epic}
\numberwithin{equation}{section}
\usepackage[margin=4%3%1.3
cm]{geometry}
\usepackage{epstopdf} 
\usepackage[colorlinks,linkcolor=blue]{hyperref}
\usepackage[capitalise,noabbrev]{cleveref}
\usepackage{todonotes}
\crefformat{equation}{(#2#1#3)}
\crefmultiformat{equation}{(#2#1#3)}{ and~(#2#1#3)}{, (#2#1#3)}{ and~(#2#1#3)}
\crefrangeformat{equation}{(#3#1#4) to~(#5#2#6)}
\usepackage[normalem]{ulem}
\usepackage{verbatim}

\allowdisplaybreaks

\theoremstyle{plain}
\newtheorem{Thm}{Theorem}[section]
\newtheorem*{Thm*}{Theorem}
\newtheorem{Lem}[Thm]{Lemma}

\newtheorem{Prop}[Thm]{Proposition}

\theoremstyle{definition}

\newtheorem{Rem}[Thm]{Remark}
\newtheorem{?}[Thm]{Problem}

\newcommand{\Do}{\mathbf{D}_{0}}
\newcommand{\Dn}{\mathbf{D}_{\neq}}

\newcommand{\p}{\partial}
\newcommand{\R}{\mathbb{R}}
\newcommand{\e}{\varepsilon}

\newcommand{\Torus}{\mathbb{T}}

\newcommand{\mr}{\mathring}

\newcommand{\ac}{\acute}

\newcommand{\abs}[1]{\left\lvert#1\right\rvert}

\newcommand{\norm}[1]{\left\lVert#1\right\rVert}

\begin{document}
	
	\begin{titlepage}
		\title{Decay rate to the planar viscous shock wave for multi-dimensional scalar conservation laws}
		\author{Lingjun Liu$^{1}$}
			%			\thanks{The research is supported by  }
		\author{Shu Wang$^{1}$}
			%			\thanks{The research is supported by  }}
		\author{Lingda Xu$^{2}$}
			%\thanks{ Corresponding author. \\ E-mail addresses: lingjunliu@bjut.edu.cn (L. Liu), wangshu@bjut.edu.cn (S. Wang), \\ 
			%xulingda@tsinghua.edu.cn (L. Xu).}}
			%		\thanks{The research is supported by  }
		
			%			\thanks{The research is supported by  }
			
		\affil{\begin{flushleft}
				\footnotesize\qquad\quad $ ^1 $ College of Mathematics, Faculty of Science, Beijing University of Technology, Beijing 100124, %P.R.
				China.\ \ \ %\\ \qquad\quad\qquad\quad
				E-mail: lingjunliu@bjut.edu.cn (L. Liu), wangshu@bjut.edu.cn (S. Wang).\\
				% E-mail: lingjunliu@bjut.edu.cn.\\
	%E-mail
	%Academy of Mathematics and Systems Science, Chinese Academy of Sciences, Beijing 100190, China.\\
				%\vspace{0.07cm}
				%\qquad \quad$^2  $ School of Mathematical Sciences, University of Chinese Academy of Sciences, Beijing 100049, China. \\
				\vspace{0.07cm}
				\qquad\quad	 $ ^2 $ Department of Applied Mathematics, The Hong Kong Polytechnic University, Hong Kong, China.\ \ \ 
				%Department of Mathematics, Yau Mathematical Sciences Center, Tsinghua University, Beijing 100084, China. \ \ \  %\\ \qquad\quad\qquad\quad 
				E-mail: %xulingda@tsinghua.edu.cn 
				lingda.xu@polyu.edu.hk (L. Xu).\\	
				\vspace{0.07cm}
				%\qquad\quad$ ^3 $ Yanqi Lake Beijing Institute of Mathematical Sciences and Applications, Beijing 101408, China. %(xulingda@tsinghua.edu.cn).
		\end{flushleft}}
		\date{}
	\end{titlepage}
	\maketitle
	\begin{abstract}
		In this paper, we study the time-decay rate toward the planar viscous shock wave for multi-dimensional (m-d) scalar viscous conservation law. We first decompose the perturbation into zero and non-zero mode, and then introduce the anti-derivative of the zero mode. Though an $L^p$ estimate and the area inequality introduced in \cite{DHS2020}, we obtained the decay rate for planar shock wave for n-d scalar viscous conservation law for all $n\geq1$. The initial perturbations we studied are small,  i.e., $\|\Phi_0\|_{H^2}\bigcap\|\Phi_0\|_{L^p}\le \varepsilon$, where $\Phi_0$ is the anti-derivative of the zero mode of initial perturbation and $\varepsilon$ is a small constant, see \cref{antiderivative}. It is noted that there is no additional requirement on $\Phi_0$, i.e., $\Phi_0(x_1)$ only belongs to $H^2(\R)$. Thus, there are essential differences from previous results, in which the initial data is required to belong to some weighted Sobolev space, cf.\cite{Goo1989,KM1985}. Moreover, the exponential decay rate of the non-zero mode is also obtained.
	\end{abstract}
	
	%Compressible Navier-Stokes equations, Couette flow, Stability threshold, Enhanced dissipation, High Reynolds number
	%\tableofcontents\
	
	%{35B40}
	%%%%%%%%%%%%%%%%%%%%%%%%%%%%%%%%%%%%%%%%%%%
	%% Introduction
	%%%%%%%%%%%%%%%%%%%%%%%%%%%%%%%%%%%%%%%%%%%
	\maketitle

	\section{Introduction}
	In this %the present
	 paper, we are concerned with the Cauchy problem of the m-d %multi-dimensional 
	 scalar conservation law as follows, %, which reads as
\begin{align}\label{eq}
	\partial_t u(x, t)+\sum_{i=1}^n \partial_i\left(f_i(u(x, t))\right)=\Delta u(x, t), \quad t>0,\ \  x%=(x_1,x_2,\cdots,x_n) 
	\in \Omega:= \mathbb{R}\times\Torus^{n-1},
\end{align}
	where the unknown function $u(x, t) \in \mathbb{R}$ is scalar, the space variable $x:=(x_1,x')=(x_1,x_2,\cdots,x_n), n \geq 2,$ $\Torus:=(\R/\mathbb{Z}),$ %$n \geq 2, u(x, t) \in \mathbb{R},
	$ \partial_i:=\frac{\partial}{\partial {x_i}}(i=1,2, \ldots, n), \triangle=\sum\limits_{i=1}^n \partial_i^2$, and $f_i(u)(i=1,2, \ldots, n)$ are smooth functions. We %should
	 further assume that the flux $f_1$ is strictly convex, i.e., $$f_1^{\prime \prime}(u) \geq c_0>0$$ for some positive constant $c_0$ and all $u \in \mathbb{R}$.

	We consider the corresponding Riemann solutions, denoted as $u^R(x_1)$, of the Riemann problem
	\begin{align}
	\left\{\begin{array}{ll}\label{rp}
	u_{t}+f(u)_{x_1}=0,    \vspace{1ex}\\
	u(0, x_1)=u^R_{0}(x_1), 
	\end{array}\right.%\qquad
	\end{align}
	where the initial data is given by
	\begin{align}\label{irp}
	u^R_0(x_1)=\left\{\begin{array}{ll}
	u_-,&x_1<0,\vspace{1ex}\\
	u_+,&x_1>0,
	\end{array}\right.\ \ \ \ \ \ \ \ \ \ \text{%where
	 $u_\pm$ are two constants.}
	\end{align}
The Riemann solutions contain two kinds of basic wave patterns: %, i.e., 
shock and rarefaction waves. In this paper, we are concerned with the shock wave case. Compared to \cref{rp}, the effect of viscosity in \eqref{eq} should be considered and the shock wave is smoothed as a smooth function, named viscous shock wave, which is a traveling wave solution to \cref{eq}. 
There are many important achievements, for example, Il'in-Oleinik \cite{IO1960} proved in the 1960s that the solution of (\ref{eq})  tends to the viscous shock wave with respect to time provided that $f(u)$ is strictly convex, i.e., $f''(u)>0$. By an additional assumption, the initial data belongs to a weighted Sobolev space, Kawashima-Matsumura \cite{KM1985} obtained the convergence rate, see also \cite{MN1994} for the case that $f(u)$ is not convex or concave. An interesting  $L^1$ stability theorem was shown %established
 in \cite{FS1998}. 
Since the pioneering %previous %pioneer
 works of Goodman \cite{Goodman1986} and Matsumura-Nishihara \cite{MN1986}, fruitful results on the asymptotic stability of traveling wave have been achieved for the systems of viscous conservation laws such as compressible Navier-Stokes system, see \cite{HM2009,KM1985,Liu1997,LZ2009,LZ2015,SX1993,Z} and the references therein. In particular,  Liu-Zeng \cite{LZ2015} obtained the pointwise estimates of viscous shock wave for conservation laws through the approximate Green function approach. 

Nevertheless, studying the decay rates toward the viscous shock wave through the basic energy method is also interesting. As far as we know,  the decay rate for scalar viscous conservation law \eqref{eq} by a weighted energy method was first obtained in \cite{KM1985}. %Since then, 
Then there have been several works on the decay properties toward the viscous shock, such as %cf.
 \cite{MN1994}, in which all of the decay rates in time depend on the decay rates of the initial data at the far fields, i.e., the initial data belongs to a weighted Sobolev space and weighted estimates are essential, %needed
  see %cf.
  \cite{KM1985}. Without this kind of additional condition, recently, Huang-Xu \cite{HX} obtained the time-decay rate toward the viscous shock wave for 1-d scalar viscous conservation law with small initial perturbations. 

%\blue{In this paper, we shall extend the result in  \cite{HX} to m-d cases. For m-d scalar conservation laws, there are also many beautiful results studied the large time behavior of shock waves. For the results derived by spectrum analysis and Green function, we refer to \cite{Goo1999,Hoff2000}. We focus on elementary energy method, \cite{Goo1989} obtained the stability by assuming initial perturbation belongs to some weighted, and \cite{Y} studied the priodic perturbations and obtained the exponential decay rate. Note that by introducing a suitable ansatz, the initial data of perturbation system is actually zero in \cite{Y}. We also refer to  \cite{HLX} for the m-d scalar conservation law with non-strictly convex, which obtained the stability of the composite wave of planar rarefaction waves and contact waves. }
{In this paper, we shall extend the result in \cite{HX} to m-d cases. For m-d scalar conservation laws, there are also many beautiful results studying the large-time behavior of shock waves. For the results derived by spectrum analysis and Green function, we refer to \cite{Goo1999,Hoff2000,Shi2016}. %We also refer to \cite{KVW2019} for more interesting results about scalar planar viscous shocks. 
We focus on the elementary energy method, \cite{Goo1989} obtained the stability by assuming initial perturbation belongs to some weighted Sobolev space. For the case of periodic perturbations, \cite{Y} studied the periodic perturbations and obtained the exponential decay rate. Note that by introducing a suitable ansatz, the initial data of the perturbation equation is zero in \cite{Y}. For the case of systems under periodic perturbations, we refer to  \cite{HuangXinXuYuan2023,HuangXuYuan2022,LiuWangXu2023}. We also refer to \cite{HLX} for the m-d scalar conservation law with non-strictly convex, which obtained the stability of the composite wave of planar rarefaction waves and contact waves.} %More multi-dimensional stability results for wave patterns can be seen \cite{HuangXinXuYuan2023,HuangXuYuan2022}.%WangWang2022}.

The main purpose of this paper is to get the decay rate in time toward the {planar viscous shock wave} for the m-d viscous conservation law \eqref{eq} without additional conditions on the initial data as in \cite{Goo1999} and \cite{Y}.  In other words, more initial perturbations can be allowed in our initial data. We first decompose the perturbation into zero and non-zero modes. Then for the zero mode, we apply the anti-derivative technique and introduce $L^{p}$ energy estimates for $p\geq2$ and obtain the decay rate for $L^p$ norm of anti-derivative, $p>2$. Next, by the area inequality \cref{aii}, we obtained the decay rate for $L^2$ norm of perturbation. For non-zero mode, Poincar\'e's inequality is available, see \cref{3333}. We use this fact to carry out a non-trivial $L^p$ energy estimate and obtain the exponential decay rate of non-zero mode. Finally, we get the decay rate of perturbation by combining these two results.

Here we are ready to state our main result. Without loss of generality, we assume that the two constants satisfy $u_-<u_+$%\red{are two constants}
. It is known that under the assumption of the so-called Lax's entropy condition, cf. \cite{M,MN1994},
    \begin{align}\label{oc}
    	h(u):=f(u)-f(u_{\pm})-s(u-u_\pm),%\qquad h'(u_-)> 0>h'(u_+),\
	 \ \ \ \ (u_-<u<u_+),
    \end{align}
    the Riemann solution to the Riemann problem (\ref{rp})-(\ref{irp}) consists of a single shock wave, cf. \cite{Smoller1983},
    \begin{align}
    u^s(x_1-st):=\left\{\begin{array}{ll}
    u_-,&x_1< st,\vspace{1ex}\\
    u_+,&x_1> st,
    \end{array}\right.
    \end{align}
    where $s$ is the shock speed and determined by the Rankine-Hugoniot condition
    \begin{align}\label{rh}
    	-s(u_+-u_-)+\left[f(u_+)-f(u_-)\right]=0.
    \end{align}
%    When $s=f'(u_-)$ or $s=f'(u_+)$, the shock wave is degnerate.  
In this paper, we consider that 
\begin{align}
h'(u_-)=f'(u_-)-s>0,\ \ \ \ \ h'(u_+)=f'(u_+)-s<0.
\end{align}

The viscous version of shock wave (viscous shock wave) %is a special solution of \eqref{eq} with the form
    \begin{align}
    	u=U(\xi),\ \ \ \ \ \xi=x_1-st, ~\lim\limits_{\xi\rightarrow\pm\infty}U(\xi)=u_\pm,
    \end{align}
    is a special solution of \eqref{eq}.
  The traveling wave $U(\xi)$ satisfies
     \begin{align}\label{a}
     \left\{\begin{array}{l}
     \left(-s U+f(U)-  U^{\prime}\right)^{\prime}=0, \vspace{1ex}\\
     {U(\pm \infty)=u_{\pm},}
     \end{array}\right.
     \end{align}
     where $':=\frac{d}{d\xi}$. We integrate \eqref{a} on $(-\infty,\xi)$ or $(\xi,+\infty)$ so that 
     \begin{align}\label{a1}
     -sU+f(U)-U'=-su_\pm+f(u_\pm),\ \ \ \xi\in\mathbb{R}.
     \end{align}
     Then the following global existence of $U(\xi)$ can be found in \cite{MN1994}. 
     \begin{Lem}\label{a2}
     Assume the Lax's entropy condition \eqref{oc} and Rankine-Hugoniot condition \eqref{rh} hold, then the equation \eqref{eq} admits a unique traveling wave solution $U(\xi)$ up to a constant shift, $\xi=x_1-st$, %\blue{which satisfies} $U'>0$, 
    , and satisfies
      $U'>0$.
     \end{Lem}
     
     %Moreover, it holds that 	$U'> 0$. 
     Let
    	$$\phi(x,t):=u(x,t)-U(\xi),$$ one has the following system
    \begin{align}\label{pe}
    	\partial_t \phi+\sum_{i=1}^n \partial_i\left[f_i(U+\phi)-f_i(U)\right]=\triangle \phi,
	\end{align}
	 with the initial data satisfies
	 \begin{align}
    	\phi_0(x):=u_0(x)-U(x_1)\in H^1(\Omega)\cap L^1(\Omega).\label{L1}
    \end{align}
      The anti-derivative of perturbation is denoted as
      \begin{align}\label{antiderivative}
      \Phi(x_1,t):=\int_{-\infty}^{x_1}\int_{\Torus^{n-1}}u(y_1,t)-U(y_1-st)dy_1,%\ \ \  \Phi_0(x_1)=\Phi(x_1,0)\in H^2(\R).
      \end{align}
      and
         \begin{align}\label{antiderivative1}
       \Phi_0(x_1)=\Phi(x_1,0)\in H^2(\R).
      \end{align}
      Without loss of generality, we assume that $\Phi(\pm\infty,0)=0$ (otherwise we can replace $U(\xi)$ by $U(\xi+a)$ with a shift $a$ determined by the initial data $u_0(x)$). 
       
The main result is 
    \begin{Thm}\label{mt}
    Under the conditions \eqref{oc}, \eqref{L1}, and \eqref{antiderivative1}, %and $\Phi_0(x_1)\in H^2(\R)$,  
    	there exists positive constants $\e_0,\ \delta_0$ such that if $\e:=\|\Phi_0(x_1)\|_{H^2}\leq \e_0$, $\delta:=\abs{u_--u_+}\leq \delta_0$, the Cauchy problem (\ref{eq}) has a unique global in time solution $u(x,t)$ satisfying
    	\begin{align}
    		{u-U \in C\left([0, \infty) ; H^{1}\right) \cap L^{2}\left([0, \infty) ; H^{2}\right).}
    	\end{align}
        Furthermore, for any $2\le p<\infty$, if ${\Phi_0(x_1)}\in{L^p(\R)},$  it holds that
    	\begin{align}\label{p}
    	%&
	\|\Phi\|_{L^p}(t)\leq Cp^\frac14\e_0(1+t)^{-\frac{p-2}{4p}}, %\quad
	\end{align}
	\begin{align}\label{p1}
    		\|u-U\|_{L^2}(t)\leq Cp^\frac18\e_0(1+t)^{-\frac{p-2}{8p}},
	\end{align}	
		\begin{align}\label{d2}
    		%&
		\|u-U\|_{L^\infty}(t)\leq Cp^\frac{1}{6}\e_0
    		(1+t)^{-\frac{(p-2)(2p+1)}{4p(3p+2)}},%\qquad 
		\end{align}
		\begin{align}\label{d21}
		\norm{\phi-\int_{\Torus^2}\phi dx}_{L^\infty}\leq C\varepsilon_0 e^{-ct},
    	\end{align}
    where $C,c$ are some positive constants.
    \end{Thm}
   % \begin{Remark}
    %	By the same idea, we can also get the decay rate for higher order term like $\Phi_{xx}$.
  %  \end{Remark} 
   \begin{Rem}
In \cite{KM1985} and \cite{MN1994}, the initial data $\Phi_0(x_1)$ belongs to a weighted Sobolev space, i.e., \begin{align}\label{add}
	\int_\R (1+x_1^2)^\frac\gamma2 \Phi^2_0(x_1)dx_1<+\infty, \ \ \ ~\gamma>0.
\end{align} 
Moreover, the decay rates obtained in \cite{KM1985} and \cite{MN1994} depend on $\gamma$. A similar requirement is needed in \cite{Goo1989}. {But} the
    additional condition \eqref{add} is removed in Theorem \ref{mt}.
     \end{Rem} 
     \begin{Rem}
     	The decay rate of $\|u-U\|_{L^2}$ is close to $(1+t)^{-\frac18}$ for %sufficiently
	 large enough $p$. Similarly, the decay rate of $\|u-U\|_{L^\infty}$ is close to $(1+t)^{-\frac16}$ and it can be improved a little as the regularity of the initial value is higher.%the initial data is more regular.%, see Remark \ref{improve} below.  
     \end{Rem}

 %    \begin{Remark}
%	The rate of $\|\Phi\|_{L^\infty}$ is close to $(1+t)^{-\frac14}$ as optimal since the same rate has been obtained by Liu  through the pointwise estimate method, cf. \cite{Liu1997}.
 %   \end{Remark}
    % \begin{Remark}
%Theorem \ref{mt} is the first work concerning with the decay rate toward the shock wave for viscous conservation law in the case that the flux $f(u)$ is not necessarily convex or concave by basic energy method.  
%     	\end{Remark}
  %  \begin{Remark}
 %	Since it is not necessary to require the right hand side of (2.5) is integrable with respect
 %	to the time t, the area inequality in Lemma 2.1 has a wide range of applications in other problems, in
 %	particular, in the study of asymptotic behaviors of solutions in the evolutionary equations.
 %   \end{Remark} 

 %\
 
   { The rest of this paper will be arranged as follows. In section \ref{prelim}, we introduce some basic lemmas which play a key role in the proof of our main theorem. %In section \ref{apri},
     A $L^p$ estimate on $\Phi$ is derived, and  Theorem \ref{mt} is proved in section 3. From \cite{IO1960,KM1985}, it is easy to know that $\|\Phi\|(t)$ is uniformly bounded by the initial data, but the $L^2$ norm $\|\Phi\|(t)$ may not tend to zero as $t\rightarrow\infty$. By a delicate $L^p$ estimate, the $L^p$ norm $(p>2)$ decays to zero with a rate of \eqref{p}. The desired decay rate \eqref{p1} and the rate \eqref{d2} are derived by making use of area inequality and Gagliardo-Nirenberg (G-N) inequality, respectively.}
   
   \ 
   
   \textbf{Notations.} 
   We denote $\|u\|_{L^p}$ by the norm of Sobolev space $L^p(\mathbb{R})$, especially $\|\cdot\|_{L^2}:=\|\cdot\|$, $C$ and $\bar{c}$ by the generic positive constants. %If $p=2$, we omit the subscript, i.e $\|\cdot\|$ is the $L^2$ norm.
   
  \section{Preliminaries}\label{prelim}
 In this section, we give some preliminaries that %, which 
 will be used in the proof of the main theorem. First we show some properties of viscous shocks as follows.
 \begin{Lem}\label{shock}\cite{SX1993,Y}
 Assume that \eqref{oc} and \eqref{rh} hold, then the viscous shock $U(x_1)$ of the problem \eqref{a1} satisfies that,  \vspace{1ex}
 
 (i) $U'(x_1)>0$ for all $x_1\in\R$; \vspace{1ex}
 
 (ii) $\delta e^{\mp C\delta x_1}\le|U(x_1)-u_\pm|\le\delta e^{\mp c\delta x_1}$ for all $x_1\in\R$ with $\pm x_1\ge 0$; \vspace{1ex}
 
 (iii)$\delta^2 e^{- C\delta |x_1|}\le|U'(x_1)|\le \delta^2 e^{- c\delta |x_1|}$ for all $x_1\in\R$; \vspace{1ex}
 
 (iv) $|U''(x_1)|\le\delta|U'(x_1)|$ for all $x_1\in\R$, \vspace{1ex}
 
 where constant $C\ge1$ is independent of $\delta$, $x_1$ and $t$.
 
 \end{Lem}

 Here %First
  we introduce  the famous %Gagliardo-Nirenberg 
  G-N inequality and the Area inequality, % established in \cite{DHS2020}
  respectively. %which reads as
 \begin{Lem}[G-N inequality \cite{KNN2004,HY}]\label{ii} Assume that $w\in L^q(\Omega)$ with $\nabla^m w\in L^r(\Omega)$, where $1\leq q,r\leq +\infty$ and $m\geq 1,$ and $w$ is periodic in the $x_i$ direction for $i=2,\cdots,n.$ Then there exists a decomposition $w(x)=\sum\limits_{k=0}^{n-1}w^{(k)}(x)$ such that each $w^{(k)}$ satisfies the $k+1$-dimensional G-N inequality, i.e.,
 	\begin{equation}\label{GN1}
 		\begin{aligned}
 			\|\nabla^{j}w^{(k)}\|_{L^p(\Omega)}\leq C\|\nabla^m w\|_{L^r(\Omega)}^{\theta_k}\|w\|_{L^q(\Omega)}^{1-\theta_k},
 		\end{aligned}
 	\end{equation}	
 	for any $0\leq j<m$ and $1\leq p\leq +\infty$ satisfying $\frac{1}{p}=\frac{j}{k+1}+(\frac{1}{r}-\frac{m}{k+1})\theta_k+\frac{1}{q}(1-\theta_k)$ and $\frac{j}{m}\leq \theta_k\leq 1.$ Hence, it holds that
 	\begin{equation}\label{GN2}
 		\begin{aligned}
 			\|\nabla^j w\|_{L^p(\Omega)}\leq C\sum_{k=0}^{n-1}\|\nabla^m w\|_{L^r(\Omega)}^{\theta_k}\|w\|_{L^q(\Omega)}^{1-\theta_k},\ \ \ (t\geq 0),
 		\end{aligned}
 	\end{equation}
 	where the constant $C>0$ is independent of $u$. Moreover, we get that for any $2\leq p<\infty$ and $1\leq q\leq p,$ it holds that	
 	\begin{equation}\label{GN3}
 		\begin{aligned}
 			\|w\|_{L^p(\Omega)}\leq C\sum_{k=0}^{n-1}\|\nabla(|w|^{\frac{p}{2}})\|_{L^2(\Omega)}^{\frac{2\gamma_k}{1+\gamma_kp}}\|w\|_{L^q(\Omega)}^{\frac{1}{1+\gamma_kp}},
 		\end{aligned}
 	\end{equation}
 	where $\gamma_k=\frac{k+1}{2}(\frac{1}{q}-\frac{1}{p})$ and the constant $C=C(p,q,n)>0$ is independent of $u$.

 \end{Lem}
%Now Based on the G-N inequality,
 We introduce the Area inequality established in \cite{DHS2020,HX}, i.e.,         
        \begin{Lem}[Area inequality%\cite{DHS2020}
        ]\label{aii}
        	Assume that a Lipschitz continuous function $f(t)\ge 0$ satisfies 
        	\begin{eqnarray}\label{lemma1}
        	f^\prime(t)  \leq  C_0(1+t)^{-\alpha},%\quad
%\text{   and      }\quad
\end{eqnarray}
and
\begin{eqnarray}\label{lemma2}
        	\int_0^t f(s)ds  \le   C_1(1+t)^\beta\ln^\gamma (1+t), ~\gamma\ge 0,
        	\end{eqnarray}
        	for some positive constants $C_0$ and $C_1$, where $0\le\beta<\alpha$.	Then if $\alpha+\beta<2$, it holds that
        	\begin{equation}\label{l1}
        	f(t) \le 2\sqrt{C_0C_1}(1+t)^{\frac{\beta-\alpha}{ 2}}\ln^{\frac{\gamma}{2}}(1+t),~t>>1. 
	\end{equation}
	Moreover, if $\beta=\gamma=0$, $f(t)\in L^1[0,\infty)$ and $0<\alpha\le2$, then 
	\begin{equation}
	f(t)=o(t^{-\frac{\alpha}{2}}) \ \ \ \text{as} \ \ \ t>>1,
        	\end{equation}
	where the index $\frac{\alpha}{2}$ is optimal.
        \end{Lem}

    \section{Proof of Theorem \ref{mt}}\label{apri}
  %     We will prove main theorem \ref{mt} with the following steps.\\
  %     \textbf{Step 1}, prove that the $L^2$ norm of $\Phi$ is bounded;\\
   %    \textbf{Step 2}, do $L^p$ estimate on $\Phi$, obtain the decay rates by an interpolation inequality;\\
   %    \textbf{Step 3}, obtain the decay rate of $L^2$ norm of $\phi$ by Dong-Huang-Su's area inequality;\\
  %     \textbf{Step 4}, use an interpolation inequality and Dong-Huang-Su's area inequality to obtain $L^2$ norm of $\phi$;\\
 %      \textbf{Step 5}, by Garglilardo-Nirenberg inequality, we finish the proof of main theorem \ref{mt}.
{ This section is devoted to proving %the proof of
 theorem \ref{mt}, the proof is based on the anti-derivative technique and $L^p$ method.}  
  
\subsection{The decomposition for $\phi$}
{To define the antiderivative of the multi-dimensional perturbation $\phi(x,t)$}, we %need to
 decompose the perturbation $\phi(x,t)$ into the principal and transversal parts. We set $\int_{\mathbb{T}^{n-1}}1dx'=1, %x'=(x_2,x_3,\cdots,x_n),
$ then we can define the %following
 decomposition $\mathbf{D}_0$ and $\mathbf{D}_{\neq}$ as follows,
\begin{equation}\label{def-decom}
	\begin{aligned}
		\mathbf{D}_{0}f:= \mr{f}:=\int_{\mathbb{T}^{n-1}}f dx',\ \ \ \ \mathbf{D}_{\neq}f:=\ac{f}:=f-\mr{f},
	\end{aligned}
\end{equation}
for an arbitrary function $f$ which is integrable on $\mathbb{T}^{n-1}$.  There are the following propositions of $\mathbf{D}_0$ and $\mathbf{D}_{\neq}$ hold for any integrable function $f$.
\begin{Prop}\label{prop-decom}\cite{HouLiuWangXu2023}
	For the projections $\Do$ and $\Dn$ defined in \cref{def-decom}, the following holds,
	
	i) $\Do\Dn f=\Dn\Do f=0$;
	
	ii) For any non-linear function $F$, one has
	\begin{align}
		\Do F(U)-F(\Do U)=O(1)F''({\Do}U)\Do\big((\Dn U)^2\big);%O(1)F''(\D_0U)\D_0\big((\Dn U)^2\big),%O(1) (\Dn U)^2,
	\end{align}
	%and similar results hold for $\tilde{U}$, $\bar{U}$, etc.
	
	iii) $\|f\|^2_{L^2(\Omega_\e)}=\|\Do f\|^2_{L^2(\R)}+\|\Dn f\|^2_{L^2(\Omega_\e)}.$
\end{Prop}
%The proof of \cref{prop-decom} is basic and we omit it.
\begin{comment}
\begin{Lem}\label{lem6}\cite{Y,HouLiuWangXu2023}%uan2023SIAM,Yuanarxiv}
{For any $p\in[1,+\infty]$, it holds that 
\begin{align}
\begin{aligned}
&\|\Do f\|_{L^p(\R)}\le\|f\|_{L^p(\Omega_\e)},\ \ %\|\Dn f\|_{L^p(\Omega_\e)}\le\|\nabla_{x'}(\Dn f)\|_{L^p(\Omega_\e)},\\
& \|\Dn f\|_{L^p(\Omega_\e)}\le\|f\|_{L^p(\Omega_\e)}+\|\Do f\|_{L^p(\R)}\le\|f\|_{L^p(\Omega_\e)}.
\end{aligned}
\end{align}}
\end{Lem}
%{The proof of Lemma \ref{lem6} can be deduced directly by making use of Cauchy's inequality and Poincar$\acute{e}$'s inequality.}
\end{comment} 
Applying $\Do$ to \eqref{pe}, we decompose the perturbation $\phi$ into the zero mode $\mr{\phi}$ and the non-zero mode $\ac{\phi}$, ($\phi=\mr{\phi}+\ac{\phi}$),
\begin{equation}\label{zm0}
\partial_{t}{\mr{\phi}}+\p_1\Big\{\Do\big(f_1( U+\phi)-f_1( U)\big)\Big\}%+f_1'(U)\p_1\mr{\phi}
=\partial_{1}^2{\mr{\phi}},%+f_1'(U)\p_1\mr{\phi}-\p_1\Big\{\Do\big(f_1( U+\phi)-f_1( U)\big)\Big\},%\qquad \Phi(x_1,0)=\Phi_0,
\end{equation}
%\begin{equation}\label{nzmi}
%\begin{case}\left\{
\begin{align}
	\left\{\begin{array}{ll}\label{nzmi}
	\p_t\ac{\phi}+\sum\limits_{i=1}^{n}\partial_{i}\bigg\{f_i( U+\phi)-f_i( U)-\Do\big(f_i( U+\phi)-f_i( U)\big)\bigg\}=\triangle\ac\phi,\vspace{1ex}\\
 \ac\phi(x,0)=0. 
	\end{array}\right.%\qquad
	\end{align}
	%\p_t\ac{\phi}+\sum\limits_{i=1}^{n}\partial_{i}\bigg\{f_i( U+\phi)-f_i( U)-\Do\big(f_i( U+\phi)-f_i( U)\big)\bigg\}=\triangle\ac\phi,\\
	%\qquad \ac\phi(x,0)=0. 
	%\end{case}.%\sum\limits_{i, j=1}^{n}  \p_1^2\ac{\phi}.
%\end{equation}
By the definition of antiderivative to $\mr{\phi}$ in \eqref{antiderivative}, we obtain 
\begin{equation}\label{zm}
	\partial_{t}{\Phi}+f_1'(U)\p_1\Phi=\partial_{1}^2{\Phi}+f_1'(U)\p_1\Phi-\Do\big(f_1( U+\phi)-f_1( U)\big),%\qquad \Phi(x_1,0)=\Phi_0,
\end{equation}
where $U$ is independent of the transverse variable $x'$. 

Theorem \ref{mt} can be derived by the following global existence theorem immediately.
   %follows from 
   \begin{Thm}[Global existence \cite{MN1994}]\label{GE} Under the conditions 
   of \ref{mt}, then the Cauchy problem \eqref{zm} with \eqref{antiderivative1} admits a unique global in time solution $\Phi(x_1,t)$ satisfying%there exist 
  \begin{align}\label{GE1}
  \|\Phi(x_1,t)\|_{H^2}^2(t)+\int_0^t\|\Phi(x_1,t)\|_{H^3}^2(\tau)d\tau\le C\e_0^{2}.%_0.
  \end{align}
  \end{Thm}

\subsection{$L^p$ estimate}
Based on the global existence in Theorem \ref{GE}, we shall establish a $L^p$ estimate for $\Phi(x_1,t)$ to obtain the decay rate.
\begin{Prop}[local existence]\label{local}
Under the assumptions of Theorem \ref{mt}, there exists constant $T_0>0$ such that the initial value problem \eqref{pe}-\eqref{L1} admits a unique smooth solution $\phi(x,t)$ on the time interval $[0,T_0]$. %if $\e$
\end{Prop}
Note that it is standard to prove the above local existence of the solution $\phi(x,t)$ for Cauchy problem \eqref{pe}-\eqref{L1} in the time interval $[0,T_0]$, we omit this proof process for brevity. 
Now we show the a priori estimates for the non-zero mode $\ac{\phi}$ as follows. Before that, we give the a priori assumptions {for any $p\in[2,+\infty)$},
\begin{align}\label{apa}
%\norm{\Phi}_{L^p}+ \|\ac\phi\|_{W^{1,p}}\leq \nu, 
\nu:=\sup_{t\in(0,T)}\Big\{\norm{\Phi}_{L^p}+ \|\ac\phi\|_{W^{1,p}}\Big\},   %\qquad\text{where $\nu<\e$ is a small positive constant. }
\end{align}
where $\nu<\e$ is a small positive constant.

\begin{Prop}[a priori estimates for the non-zero mode $\ac{\phi}$]\label{pet1}
	%\blue{If the solution }
	{Assume that} $\ac{\phi}(x,t)$ is the local smooth solution, then for any $p\in [2,+\infty)$, %any $ t\in[0,T]$ and positive constant $T\le T_0$, }%\blue{obtained in Proposition \ref{let},  then for $t\in[0,T],$} 
	we have
	\begin{equation}\label{p3.2}
		\begin{aligned}
			&\|\ac{\phi}(\cdot,t)\|_{W^{1,p}({\Omega})}\leq C\e_0 e^{-\bar{c}t},\quad  \forall p\in [2,+\infty), \ \ t\in[0,T], %\ \ T\le T_0.
		\end{aligned}
	\end{equation}	
	where positive constant {$T\le T_0$ is arbitrary.}
\end{Prop}
The proof of Proposition \ref{pet1} is divided into the following lemmas.
\begin{Lem}[the basic $L^p$ estimate for $\ac{\phi}, 2\leq p<+\infty$]\label{L4.1}
	Under the same assumptions of \cref{pet1}, it holds that
	\begin{equation}\label{l4.01}
		\begin{aligned}
			\frac{d}{dt}\left\|\ac{\phi}\right\|_{L^p}^p+\left\|\nabla\left(|\ac{\phi}|^{\frac{p}{2}}\right)\right\|_{L^2}^2 \leq C(\varepsilon_0+ \delta+\nu)\left\|\nabla\ac{\phi}\right\|_{L^2(\Omega)}^p.%C\varepsilon e^{-\bar{c}t\cdot p}+(C\varepsilon e^{-\bar{c}t}+\nu)\|\ac{\phi}\|_{L^p}^p.
		\end{aligned}
	\end{equation}		
\end{Lem}
\begin{proof}
\begin{comment}
	%Since the cases of the first-order derivative are similar, we only
	We first prove 
	\begin{align}\label{tde}
	\norm{\ac{\phi}}_{L^p}\leq C\e e^{-\bar{c}t}.
	\end{align}
	\end{comment}
	For any $p\in[2,+\infty)$, multiplying \eqref{nzmi}$_1$ by $|\ac{\phi}|^{p-2}\ac{\phi}$ and then integrating the resulting equation on $\Omega$,  we have
	\begin{equation}\label{l4.11}
		\begin{aligned}
			&\frac{1}{p}\p_t\norm{\ac{\phi}}_{L^p(\Omega)}^p+(p-1)\sum\limits_{i=1}^{n} \int_{\Omega}|\ac{\phi}|^{p-2}\p_i\ac{\phi}\p_i\ac{\phi}dx\\
			&=\sum\limits_{i=1}^{n}\int_{\Omega}\bigg\{f_i( U+\phi)-f_i( U)-\Do\big(f_i( U+\phi)-f_i( U)\big)\bigg\}\partial_{i}(|\ac{\phi}|^{p-2}\ac{\phi})dx.
		\end{aligned}
	\end{equation}
	As for the first term on the right-hand-side of \eqref{l4.11} satisfying, remember $\phi=\mr{\phi}+\ac{\phi},$ 
	\begin{equation}\label{l4.14}
		\begin{aligned}
			\big\{f_i( U+\phi)-f_i( U)-\Do\big(f_i( U+\phi)-f_i( U)\big)\big\}
			=\big(f_i'( U)\ac{\phi}\big)+O(1)\big(\ac{\phi}^2+\ac{\phi}\mr{\phi}^2\big).
		\end{aligned}
	\end{equation}
	%where $\phi=\mr{\phi}+\ac{\phi}.$
Moreover, one has
	\begin{equation}\label{l4.16}
		\begin{aligned}
			I_1:&=\int_{\Omega}\bigg(f_i'( U)\ac{\phi}\bigg)\partial_{i}(|\ac{\phi}|^{p-2}\ac{\phi})dx\\
			&=\int_{\Omega}\partial_{i}\bigg(\frac{p-1}{p}f_i'( U)|\ac{\phi}|^{p}\bigg)-\frac{p-1}{p}f^{''}_1( U)|\ac{\phi}|^{p}\partial_{{1}} Udx,\vspace{1ex}\\
			I_2:=&\int_{\Omega}(\ac{\phi}^2+\ac{\phi}\mr{\phi}^2)\partial_{i}(|\ac{\phi}|^{p-2}\ac{\phi})dx\leq O(\nu)\big(\left\|\nabla\big(|\ac{\phi}|^{\frac{p}{2}}\big)\right\|_{L^2}^2+\|\ac{\phi}\|_{L^p}^p\big).
		\end{aligned}
	\end{equation}
	By \cref{ii} and the fact that $\abs{\p_1U}<\delta^2$ in lemma \ref{shock}, one has
	\begin{align}\label{I4-15}
		\norm{\ac{\phi}}_{L^p}^p\leq C\|\nabla\big(|\ac{\phi}|^{\frac{p}{2}}\big)\|_{L^2(\Omega)}^{\frac{2\gamma_kp}{1+\gamma_kp}}\|\ac{\phi}\|_{L^2(\Omega)}^{\frac{p}{1+\gamma_kp}}
		\leq \nu \|\nabla\big(|\ac{\phi}|^{\frac{p}{2}}\big)\|_{L^2}^2+C\|\ac{\phi}\|_{L^2}^p.
	\end{align}Combining \eqref{l4.14}-\eqref{I4-15}, one has
	\begin{equation}\label{3333}
		\begin{aligned}
			\frac{d}{dt}\|\ac{\phi}\|_{L^p}^p+\|\nabla\big(|\ac{\phi}|^{\frac{p}{2}}\big)\|_{L^2}^2\leq C(\varepsilon_0+ \delta+\nu)\|\ac{\phi}\|_{L^2(\Omega)}^p\leq C(\varepsilon_0+ \delta+\nu)\|\nabla\ac{\phi}\|_{L^2(\Omega)}^p,
		\end{aligned}
	\end{equation}
	where we have used Poincar\'e's inequality since $\int_{\Omega}\ac{\phi}dx=0$. 
	\end{proof}
	
	Furthermore, for $p=2$, one can {directly calculate} %\blue{apply Gronwall's inequality to get} 
	 the exponential decay rate. For $p>2$, one should use the obtained result of $p=2$ and \eqref{3333} %\blue{Gronwall's inequality}
	 %G-N inequality 
	 to obtain the desired decay rate. The detailed mathematics analysis %process
	 is described as follows.
	 \begin{Lem}[Time decay estimate for $\ac\phi$, $2\le p<+\infty$] \label{tde1}
	 \begin{align}\label{tde}
	 \|\ac\phi(\cdot,t)\|_{L^p(\Omega)}\le C\e_0 e^{-\bar ct},\ \ \ \forall p\in[2,+\infty).
	 \end{align}
	 \end{Lem}
	\begin{proof} 
	
	{\bf Case 1.} When $p=2$, from \eqref{l4.01} and making use of  Poincar\'e's inequality, 
	%it yields 
	one obtains %directly
	 that
	\begin{align}\label{l12l}
	%\|\ac{\phi}\|_{L^2}^2\le C\|\nabla\ac\phi\|_{L^2}^2
	\frac{d}{dt}\|\ac{\phi}\|^2+\|\ac\phi\|^2+\|\nabla\big(|\ac{\phi}|\big)\|^2\leq C\e_0 e^{-\bar ct}.
	\end{align}
	Then we get \eqref{tde} for $p=2$ quickly.
	
	{\bf Case 2.} When $p>2$, by \eqref{3333} %the time decay estimate as $p=2$ 
	and \eqref{l12l}, %G-N inequality, 
	one has
	\begin{equation}\label{l3333}
		\begin{aligned}
			\frac{d}{dt}\|\ac{\phi}\|_{L^p}^p+\|\nabla\big(|\ac{\phi}|^{\frac{p}{2}}\big)\|_{L^2}^2\leq C(\varepsilon+ \delta+\nu)\|\ac{\phi}\|_{L^2(\Omega)}^p\leq C\e_0 e^{-\bar ct}.%(\varepsilon+ \delta+\nu)%\|\nabla\ac{\phi}\|_{L^2(\Omega)}^p,
		\end{aligned}
	\end{equation}
	%\begin{align}
	%\|\ac\phi\|^p\le
	%\end{align}
	Then it holds that 
	\begin{align}
	%\|\ac{\phi}\|_{L^2}^2\le C\|\nabla\ac\phi\|_{L^2}^2
	\frac{d}{dt}\|\ac{\phi}\|_{L^p}^p+\|\ac\phi\|_{L^p}^p+\|\nabla\big(|\ac{\phi}|^{\frac{p}{2}}\big)\|^2\leq C\e_0 e^{-\bar ct}.
	\end{align}
	Thus we get \eqref{tde} for $p\in(2,+\infty).$

\end{proof}

%\begin{comment}	
\begin{Lem}[Time decay estimate for $\nabla\ac{\phi}, 2\leq p<+\infty$]\label{L4.20.}
	Under the same assumptions of \cref{pet1}, it holds that
	\begin{equation}\label{l4.001}
		\begin{aligned}
			\left\|\nabla\ac{\phi}(\cdot,t)\right\|_{L^p(\Omega)} \leq C\varepsilon_0 e^{-\bar ct}, \ \ \ \ \forall p\in[2,+\infty).%C\varepsilon e^{-\bar{c}t\cdot p}+(C\varepsilon e^{-\bar{c}t}+\nu)\|\ac{\phi}\|_{L^p}^p.
		\end{aligned}
	\end{equation}		
\end{Lem}
%\end{comment}
\begin{proof}
Taking the derivative on \eqref{nzmi}$_1$ with respect to $x_k,$ $k=1,2,\cdots,n$, multiplying the resulting equation by $|\p_k\ac\phi|^{p-2}\p_k\ac\phi$, %here we only estimate the case of $k=1$ since these two cases are similar and easier. Then
  when $k=1,$ one has
\begin{align}\label{nabla1}
\begin{aligned}
%\frac{1}{p}\p_t\|\p_k\ac\phi\|_{L^p}^p+
%&\frac{1}{p}\p_t\norm{\p_k\ac{\phi}}_{L^p(\Omega)}^p+(p-1)\sum\limits_{i=1}^{n} \int_{\Omega}|\p_k\ac{\phi}|^{p-2}\p_{ik}\ac{\phi}\p_{ik}\ac{\phi}dx\\
			%&=\blue{\sum\limits_{i=1}^{n}\int_{\Omega}\bigg\{f_i( U+\phi)-f_i( U)-\Do\big(f_i( U+\phi)-f_i( U)\big)\bigg\}}
			\frac{1}{p}\p_t&\norm{\p_1\ac{\phi}}_{L^p(\Omega)}^p+(p-1)\sum\limits_{i=1}^{n} \int_{\Omega}|\p_1\ac{\phi}|^{p-2}\p_{i1}\ac{\phi}\p_{i1}\ac{\phi}dx\\
			&=\sum\limits_{i=1}^{n}\int_{\Omega}\bigg\{\Big[f'_i( U)\p_1\ac\phi-\Do\big(f'_i( U)\p_1\ac\phi\big)\Big]
+\Big[\big(f'_i(U+\phi)-f'_i(U)\big)(\p_1 U+\p_1\ac\phi)\Big]\\
			&\qquad-\Do\Big[\big(f'_i(U+\phi)-f'_i(U)\big)(\p_1U+\p_1\ac\phi)\Big]\\
			&\qquad+\Big[f'_i( U+\phi)\p_1\mr\phi-\Do\big(f_i( U+\phi)\p_1\mr\phi\big)\Big]\bigg\}\partial_{i}\left(|\p_1\ac{\phi}|^{p-2}\p_1\ac{\phi}\right)dx:=%\sum\limits_{i=1}^{n}
			J,
\end{aligned}
\end{align}
and for $2\le k\le n,$
\begin{align}
\begin{aligned}
\frac{1}{p}\p_t&\norm{\p_k\ac{\phi}}_{L^p(\Omega)}^p+(p-1)\sum\limits_{i=1}^{n} \int_{\Omega}|\p_k\ac{\phi}|^{p-2}\p_{ik}\ac{\phi}\p_{ik}\ac{\phi}dx=\sum\limits_{i=1}^{n}\int_{\Omega}\bigg\{\Big(f'_i( U)\p_k\ac\phi\Big)\\
			&+\Big[\big(f'_i(U+\phi)-f'_i(U)\big)(\p_k U+\p_k\ac\phi+\p_k\mr\phi)\Big]%\\
			%&\qquad-\Do\Big[\big(f'_i(U+\phi)-f'_i(U)\big)(\p_1U+\p_1\ac\phi)\Big]\\
			%&\qquad+\Big[f'_i( U+\phi)\p_1\mr\phi-\Do\big(f_i( U+\phi)\p_1\mr\phi\big)\Big]
			\bigg\}\partial_{i}\left(|\p_k\ac{\phi}|^{p-2}\p_k\ac{\phi}\right)dx.
			\end{aligned}
			\end{align}
Here we only estimate the case of $k=1$ since these two cases are similar and easier for $2\le k\le n$.	
		
The term $J$ can be divided into three terms as follows.
Similar to $I_1$ in \eqref{l4.16}, we have 
\begin{align}
\begin{aligned}
&J_1:=\int_{\Omega}\left(f'_i(U)\p_1\ac\phi\right)\p_i\left(|\p_1\ac\phi|^{p-2}\p_1\ac\phi\right)dx\le C(\e+\delta)\|\p_1\ac\phi\|_{L^p}^p.\\
\end{aligned}
\end{align}
Similar to $I_2$ in \eqref{l4.16}, by making use of H{\"o}lder inequality and  \eqref{tde} in lemma \ref{tde1}, we get 
\begin{align}
\begin{aligned}
J_2:&=O(1)\int_{\Omega}\left(\ac U\p_1\mr\phi+\ac U\p_1\ac\phi+\ac\phi\p_1\mr\phi+\mr\phi\p_1\mr\phi+\ac\phi\p_1\ac\phi\right)\p_i\left(|\p_1\ac\phi|^{p-2}\p_1\ac\phi\right)dx\\
&\le C\nu\left(\left\|\nabla\left(|\p_1\ac\phi|^{\frac{p}{2}}\right)\right\|^2+\|\p_1\ac\phi\|^p_{L^p}\right)+C\e\left(\|\p_1\mr\phi\|_{L^p}^p+\|\p_1\mr\phi\|_{L^{2p}}^p+\|\p_1\ac\phi\|_{L^p}^p\right),
\end{aligned}
\end{align}
where 
\begin{align}
\begin{aligned}
\int_\Omega\ac\phi\p_1\mr\phi|\p_1\ac\phi|^{p-2}\p_{i1}\ac\phi dx&\le C\|\ac\phi\|_{L^{2p}}\|\p_1\mr\phi\|_{L^{2p}}\|\p_1\ac\phi\|_{L^p}^{\frac{p-2}{2}}\left\|\nabla\left(|\p_1\ac\phi|^{\frac{p}{2}}\right)\right\|\\
&\le \nu\left(\left\|\nabla\left(|\p_1\ac\phi|^{\frac{p}{2}}\right)\right\|^2+\|\p_1\ac\phi\|_{L^p}^p\right)+C\|\ac\phi\|_{L^{2p}}^p\|\p_1\mr\phi\|_{L^{2p}}^p.
\end{aligned}
\end{align}

Because of the property of viscous shocks in lemma \ref{shock}, the estimate of $J_3$ is easier, %where 
\begin{align}
\begin{aligned}
J_3:&=O(1)\int_{\Omega}\left[\left(\ac\phi+\mr\phi+\mr\phi\ac\phi\right)\p_1\ac U+\left(\ac\phi+\mr\phi\ac U+\mr\phi\ac\phi\right)\p_1\mr U\right]\p_i\left(|\p_1\ac\phi|^{p-2}\p_1\ac\phi\right)dx\\
&\le C(\nu+\e+\delta)\left(\left\|\nabla\left(|\p_1\ac\phi|^{\frac{p}{2}}\right)\right\|^2+\|\p_1\ac\phi\|^p_{L^p}+\|\mr\phi\|_{L^p}^p\right)+C\|\ac\phi\|_{L^{p}}^p,%+\|\p_1\ac\phi\|_{L^p}^p.
\end{aligned}
\end{align}
where 
\begin{align}
\begin{aligned}
\int_\Omega&\left(\ac\phi+\mr\phi\ac U\right)|\p_1\ac\phi|^{p-2}\p_{i1}\ac\phi dx\\
&\le \nu\left(\left\|\nabla\left(|\p_1\ac\phi|^{\frac{p}{2}}\right)\right\|^2+\|\p_1\ac\phi\|^p_{L^p}\right)+C\|\ac\phi\|_{L^{p}}^p+C(\e+\nu)\|\mr\phi\|_{L^p}^p.
\end{aligned}
\end{align}
Therefore, it yields
\begin{align}\label{nabla2}
\begin{aligned}
			\frac{d}{dt}\norm{\p_1\ac{\phi}}_{L^p}^p+\left\|\nabla\left(|\p_1\ac{\phi}|^{\frac{p}{2}}\right)\right\|^2\le C(\nu+\delta+\e_0)\|\p_1\ac\phi\|_{L^p}^p+C\e_0 e^{-\bar ct}.
\end{aligned}
\end{align}
In order to get \eqref{l4.001} for $k=1$, we only replace $\ac\phi$ with $\p_1\ac\phi$ and then follow the proof steps in Lemma \ref{tde1}.
\end{proof}

Now we begin to use the $L^p$ method %estimates
 to study the decay rate for the antiderivative $\Phi(x_1,t)$ in \eqref{zm}.%zero mode $\mr\phi$
       \begin{Prop}\label{1pp}
        Under the conditions of Theorem \ref{mt}, it holds that, for $2< p< \infty$,
        	\begin{align}\label{lp}
        	\|\Phi\|_{L^p}\leq Cp^\frac14\varepsilon_0(1+t)^{-\frac{p-2}{4p}},
        	\end{align}
	where $C$ is independent of $p$.
       %	\begin{align}\label{lpeq}
       %	(1+t)^c\|\Phi(t)\|_{L^p}^p&+\int_0^t(1+\tau)^c\left\|\partial_x\left(|\Phi|^{\frac{p}{2}}\right)\right\|^2d\tau\\
       %	&\leq C\|\Phi_{0}\|^p+C\left(\|\Phi_{0}\|^2+1\right)^{\frac{p}{2}}(1+t)^{c-\frac{(p-2)}{4}}.\nonumber
       %	\end{align}
       \end{Prop}
   \begin{proof}
   %	Since we only assume the Oleinik entropy condition, so $f''(u)\neq0$ does not hold on the whole space. Thus, We need an weighted energy estimate for ensuring the right sign.
   		As in \cite{MN1994}, we choose the weight function $w(u)$ as 
   	\begin{align}
   		w(u)=\left\{\begin{array}{ll}
   			{-\frac{\left(u-u_{-}\right)\left(u-u_{+}\right)}{h(u)},} & {\left(u_{-}<u<u_{+}\right),}\vspace{1ex} \\
   			{-\frac{u_{\pm}-u_{\mp}}{f^{\prime}\left(u_{\pm}\right)-s},} & {\left(u=u_{\pm}\right).}
   		\end{array}\right.
   		%\qquad 
   		%C^{-1}<w<C,\ \ \ \ \ \ \ \ \ \ \ \ (hw)''=-2,
   	\end{align}
	By %the condition 
	\eqref{oc}, there exists a positive constant $C$ such that
	\begin{align*}
	C^{-1}<w<C,\ \ \ \ \ \ \ \ \ \ \ \ (hw)''=-2,
	\end{align*}
   	   	where $(hw)''$ means $\frac{d^2}{dU^2}(h(U)w(U)).$ 
For any $2<%\leq
 p<\infty$, multiplying (\ref{zm}) by $w|\Phi|^{p-2}\Phi$, and following the same line as in \cite{MN1994}, see also \cite{M}, we arrive at
   	\begin{align}
   		\frac{d}{d\tau}\int &\frac{1}{p} w|\Phi(\tau)|^{p} dx_1-\int \frac{1}{p}(h w)^{\prime \prime}|\Phi|^{p} U^{\prime} d x_1   +(p-1)\int   w\left|\p_1\Phi\right|^{2}|\Phi|^{p-2} d x_1 \\
		&= \int w |\Phi|^{p-2}\Phi N d x_1, \nonumber
   	\end{align}
	where
   	\begin{align}
   		\begin{aligned}
   		N:=&f_1'(U)\p_1\Phi-(f_1(U+\p_1\Phi)-f_1(U))-\big(\Do(f_1( U+\phi)\\
		&\qquad-f_1( U))-(f_1(U+\p_1\Phi)-f_1(U))\big)\\
   		=&\int_0^1f''_1(U+\theta\mr{\phi})\theta d\theta \mr
		{\phi}^2+O(1)\mathbf{D}_{0}\left(\ac{\phi}\mr{\phi}^2+\ac{\phi}^2\right)\\
		=&:Q(U,\mr{\phi})\mr
		{\phi}^2+O(1)\mathbf{D}_{0}\left(\ac{\phi}\mr{\phi}^2+\ac{\phi}^2\right).
   		   		\end{aligned}
   	\end{align}
Then by $U'>0$, one has
	\begin{align}\label{ineq1}
 	\frac{d}{d\tau}\int w|\Phi(\tau)|^{p} dx_1+\int   w\|\p_1\big(|\Phi|^{\frac{p}{2}}\big)\|^{2}d x_1 \le C(\varepsilon+\delta+\nu)e^{-\bar{c} t}.
 	\end{align}
To get the decay rate \eqref{lp}, 
  multiplying \eqref{ineq1} by $(1+\tau)^\sigma$, and then integrating the resulting equation on $(0,t)$, we get
  \begin{align}\label{lpee}
  	(1+t)^\sigma&\|\Phi(t)\|_{L^p}^p +\int_0^t(1+\tau)^\sigma\left\|\partial_1\left(|\Phi|^{\frac{p}{2}}\right)\right\|^2d\tau\\
	&\leq \|\Phi_0\|_{L^p}^p+\sigma\int_{0}^{t}(1+\tau)^{\sigma-1}\|\Phi(\tau)\|_{L^p}^pdt.\nonumber
  	\end{align}
By the Sobolev's inequality, we have
\begin{align}
\|\Phi\|_{L^p}^p\le \|\Phi\|^2\|\Phi\|_{L^\infty}^{p-2},
\end{align}
\begin{align}
\|\Phi\|_{L^\infty}^p\le 2\|\Phi\|^{\frac{p}{2}}_{L^p}\left\|\p_1\left(|\Phi|^{\frac{p}{2}}\right)\right\|.%_{L^\infty}^{p-2}
\end{align}
Then it yields%Note that
\begin{align}
	\begin{aligned}
\|\Phi\|_{L^p}^p\le%& \|\Phi\|^2\|\Phi\|_{L^\infty}^{p-2}\leq 
2^{\frac{2(p-2)}{p+2}}\|\Phi\|^{\frac{4p}{p+2}}\left\|\p_1(|\Phi|^{\frac{p}{2}})\right\|^\frac{2(p-2)}{p+2},%\vspace{1ex}\\
  	%\leq& \frac{p-2}{p+2}(1+t)\left\|\partial_1\left(|\Phi|^{\frac{p}{2}}\right)\right\|^2+\frac{4}{p+2}2^\frac{p-2}{2}\sigma^\frac{p+2}{4}(1+t)^{-\frac{p-2}{4}}\|\Phi\|_{L^2}^{p}.
  \end{aligned}
\end{align}
and from Cauchy's inequality, it holds that 
\begin{align}\label{eqsi}
\begin{aligned}
&(1+t)^\sigma\|\Phi(t)\|_{L^p}^p\\
&\le\frac{p-2}{p+2}(1+t)^\sigma\left\|\partial_1\left(|\Phi|^{\frac{p}{2}}\right)\right\|^2+\frac{4}{p+2}2^\frac{p-2}{2}\sigma^\frac{p+2}{4}(1+t)^{\sigma-\frac{p-2}{4}}\|\Phi\|_{L^2}^{p}.
\end{aligned}
\end{align}
  Choosing $\sigma=\frac{p+2}{4}$, we get \cref{lp}.
   		%Letting $p\to \infty$, one has that  
   		%	\begin{align}\label{rate1}
   		%	\|\Phi\|_{L^\infty}=\lim_{p\to \infty}\|\Phi\|_{L^p}\leq C\e(1+t)^{-\frac{1}{4}},
   		%	\end{align}
   		%	and the proof is completed.
      \end{proof}
      \subsection{Decay of $\|\p_1\Phi%_{x_1}
      \|=\|\mr{\phi}\|$}
  \begin{Prop}\label{p2}
  	Under the conditions of Theorem \ref{mt}, it holds that, 
  	\begin{align}\label{rate2}
  		\|\mr{\phi}\|_{H^1(\Omega)}\leq Cp^\frac18\e_0(1+t)^{-\frac{p-2}{8p}}.
  	\end{align}
  \end{Prop}
     \begin{proof}
	  Multiplying \eqref{zm0} %$\mathbf{D}_{0}\{(\ref{pe})\}$
	   by $\mr\phi$ and then integrating the resulting equation on $\mathbb{R}$ with respect to $x_1$, we have
	  \begin{align}\label{3.17}
	     \frac{1}{2}\frac{d}{dt}\|\mr{\phi}\|^2+\int_{\mathbb{R}}\p_1\left[f'_1(U)\mr{\phi}\right]\mr{\phi} dx_1+ \|\p_1\mr{\phi}\|^2=\int_{\mathbb{R}}\p_1N(x_1,t)\mr{\phi} dx_1.
	  \end{align}
	  Then we have
	  \begin{align}\label{N}
	  \begin{aligned}
	  	&\abs{\int_{\mathbb{R}}\p_1N(x_1,t)\mr{\phi} dx_1}\\
		&\le\abs{\int_{\mathbb{R}}Q(U,\mr{\phi})\p_1\left(\frac{2\mr{\phi}^3}{3}\right)+\p_1Q(U,\mr{\phi})\mr{\phi}^3dx_1}+\abs{O(1)\int_{\R}\p_1\mr{\phi}\Do(\ac{\phi}\mr{\phi}+\ac{\phi}^2)dx_1}\\
	  	&=\abs{\frac{1}{3}\int_{\mathbb{R}}\left[Q_{U}U'+Q_{\phi}\p_1\mr{\phi}\right]\mr{\phi}^3dx_1}+O(\varepsilon+\delta+\nu)\left(e^{-\bar{c}t}+\norm{\p_1\mr{\phi}}^2_{L^2}\right)\\
	 	&\le C \|\mr{\phi}\|^2_{L^\infty}+C(\varepsilon+\delta+\nu)\left(\|\p_1\mr{\phi}\|^2+e^{-\bar{c}t}\right),
		\end{aligned}
		\end{align}
		\begin{align}
	   \left|\int_{\mathbb{R}}\p_1\left[f'_1(U)\mr{\phi}\right]\mr{\phi} dx_1\right|=&\left|\int_{\mathbb{R}}f'_1(U)\p_1\left(\frac{\mr{\phi}^2}{2}\right)dx_1\right|\le C\int_{\mathbb{R}}U'\mr{\phi}^2dx_1\leq C \|\mr{\phi}\|^2_{L^\infty}.
	   \end{align}
%It remains to deal with $\int_{\mathbb{R}}U'\phi^2dx$, which is estimated as 
%	  \begin{align}
%	  	\int_{\mathbb{R}}U'\phi^2dx\leq C \|\phi\|^2_{L^\infty}
%	  \end{align} 
	  By the G-N %Garglilardo-Nirenberg
	   inequality \cref{GN2}, we have 
	  \begin{align}\label{3.22}
	  \|\mr\phi\|^2_{L^\infty}=\|\p_1\Phi\|_{L^\infty}^2\leq C\|\p_1\mr{\phi}\|^{\frac{4(p+1)}{3p+2}}\|\Phi\|_{L^p}^{\frac{2p}{3p+2}}\leq  4\|\p_1\mr{\phi}\|^{2}+C\|\Phi\|_{L^p}^{2}.
	  \end{align}
%	  where $a=\frac{2p+2}{3p+2}$, and in the last inequality, we have used Holder inequality with $a+1-a=1.$	  
From  \eqref{3.17}-\eqref{3.22},	 we get 
	 \begin{align}\label{2l}
	 	\frac{d}{dt}\|\mr{\phi}\|^2\leq Cp^\frac12\e^2(1+t)^{-\frac{p-2}{2p}}.
	 \end{align}
	 Due to the proposition \ref{1pp}, we know
	 \begin{align}
	 \int_0^\infty\|\mr{\phi}\|^2(\tau)d\tau\le C\e_0^2.
	 \end{align}
Then we conclude from the area inequality, i.e., Lemma \ref{aii}, that 
	 \begin{align}\label{2d}
	 	\|\mr{\phi}\|^2\leq Cp^\frac14\e_0^2(1+t)^{-\frac{p-2}{4p}}.
	 \end{align}
	% and finally obtain the decay rate \eqref{rate2}. 
	
	%On the other hand,
	Now we are ready to estimate $\|\p_1\mr\phi%_{x_1}
	\|$. Multiplying \eqref{zm0} %$\mathbf{D}_{0}\{(\ref{pe})\}$
	   by $-\p_1^2\mr\phi%_{x_1x_1}
	   $ and then integrating the resulting equation on $\mathbb{R}$ with respect to $x_1$, we have
	  \begin{align}\label{3.17.}
	     \frac{1}{2}\frac{d}{dt}\|\p_1\mr{\phi}%_{x_1}
	     \|^2+ \|\p^2_1\mr{\phi}\|^2=\int_{\mathbb{R}}\p_1\left[f'_1(U)\mr{\phi}\right]\p_1^2\mr{\phi}%_{x_1x_1}
	      dx_1-\int_{\mathbb{R}}\p_1N(x_1,t)\p_1^2\mr{\phi}%_{x_1x_1} 
	      dx_1.
	  \end{align}
	  By a direct computation, we have
	  \begin{align}
	  \begin{aligned}
	  \left|\int_{\mathbb{R}}\p_1\left[f'_1(U)\mr{\phi}\right]\p_1^2\mr{\phi}%_{x_1x_1}
	   dx_1\right|&\le\frac{1}{8}\|\p_1^2\mr{\phi}%_{x_1x_1}
	   \|^2+C\left(\|\mr\phi\|_{L^\infty}^2+\|\p_1\mr\phi%_{x_1}
	   \|_{L^\infty}^2\right)\\
	  &\le\frac{1}{4}\|\p_1^2\mr{\phi}%_{x_1x_1}
	  \|^2+Cp^{\frac{1}{2}}\e_0^2(1+t)^{-\frac{p-2}{2p}},
	  %=&\left|\int_{\mathbb{R}}f'_1(U)\p_1\left(\frac{\mr{\phi}^2}{2}\right)dx_1\right|\le C\int_{\mathbb{R}}U'\mr{\phi}^2dx_1\leq C \|\mr{\phi}\|^2_{L^\infty}.
	  \end{aligned}
	   \end{align}
where in fact that
	\begin{align}
	\|\mr\phi\|_{L^\infty}^2\le C\|\p_1^2\mr\phi%_{x_1x_1}
	\|^{\frac{4(p+1)}{5p+2}}\|\Phi\|_{L^p}^{\frac{6p}{5p+2}}\le\frac{1}{16}\|\p_1^2\mr\phi%_{x_1x_1}
	\|^2+Cp^{\frac{1}{2}}\e_0^2(1+t)^{-\frac{p-2}{2p}},
	\end{align}
	\begin{align}
	\|\p_1\mr\phi%_{x_1}
	\|_{L^\infty}^2\le C\|\p_1^2\mr\phi%_{x_1x_1}
	\|^{\frac{4(2p+1)}{5p+2}}\|\Phi\|_{L^p}^{\frac{2p}{5p+2}}\le\frac{1}{16}\|\p_1^2\mr\phi%_{x_1x_1}
	\|^2+Cp^{\frac{1}{2}}\e_0^2(1+t)^{-\frac{p-2}{2p}}.
	\end{align}
	The last term on the right-hand-side of \eqref{3.17.} yields that 
	 \begin{align}\label{N.}
	  \begin{aligned}
	  	\abs{\int_{\mathbb{R}}\p_1N(x_1,t)\p_1^2\mr{\phi}%_{x_1x_1}
		 dx_1}&\le\frac{1}{8}\|\p_1^2\mr\phi%_{x_1x_1}
		 \|^2+\int_{\mathbb{R}}|\p_1N|%_{x_1}
		 ^2(x_1,t)dx_1\\
		&\le\frac{1}{8}\|\p_1^2\mr\phi%_{x_1x_1}
		\|^2+Cp^{\frac{1}{2}}\e_0^2(1+t)^{-\frac{p-2}{2p}}.
		\end{aligned}
		\end{align}
	Thus we have
	\begin{align}\label{L.}
	\frac{d}{dt}\|\p_1\mr{\phi}%_{x_1}
	\|^2\leq Cp^{\frac{1}{2}}\e_0^2(1+t)^{-\frac{p-2}{2p}},
	 \end{align}
	 and $\|\p_1\mr\phi%_{x_1}
	 \|^2\in L^1(0,+\infty)$ by Theorem \ref{GE}. Then using the area inequality again, one gets
	 \begin{align}
	 \|\p_1\mr{\phi}%_{x_1}
	 \|^2\leq Cp^{\frac{1}{4}}\e_0^2(1+t)^{-\frac{p-2}{4p}}.
	 \end{align}
	 Finally, we obtain the decay rate \eqref{rate2}.
	 \end{proof}
%	 \begin{Rem}
%	 The Gagliardo-Nirenberg inequality 
%	 \begin{equation}
%	 \|\phi\|_{L^2}\le \|\Phi\|_{L^p}^{1-\lambda}\|\phi_x\|^\lambda, ~~\lambda=\frac{p+2}{3p+2}
%	 \end{equation}
%	 holds only for $p\le 2$, while $\|\Phi\|_{L^2}$ is known uniformly bounded. Thus it is difficult to obtain the decay rate of $\|\phi\|_{L^2}$ through the GN inequality, and the area inequality can provide  the rate \eqref{rate2}.
%	 \end{Rem}
% Next we calculate the decay rates of $L^p$ norms of $\phi$, for $p>2$. 

%We are ready to prove Theorem \ref{mt}. 

%\subsection{Decay of $\|\Phi_{x_1x_1}\|=\|\mr{\phi}_{x_1}\|$}
%\begin{Prop}
%Under the conditions of Theorem \ref{mt}, it holds that 
%\begin{align}

%\end{Prop}

\begin{proof}[\bf Proof of Theorem \ref{mt}]
It remains to show \eqref{d2}, which can be achieved from the G-N inequality and the decay rate \eqref{rate2}, i.e., 
\begin{align*}
\|\mr{\phi}\|_{L^\infty}\le C\|\Phi\|^\frac{p}{3p+2}_{L^p}\|\p_1\mr{\phi}\|^\frac{2(p+1)}{3p+2}\le  C\e_0
p^\frac{1}{6}(1+t)^{-\frac{(p-2)(2p+1)}{4p(3p+2)}},%\qquad
\end{align*}
\begin{align*}
\|\ac{\phi}\|_{L^{\infty}}\le C\sum_{k=0}^{n-1}\|\nabla \ac{\phi}\|_{L^{r_k}(\Omega)}^{\theta_k}\|\ac{\phi}\|_{L^{q_k}(\Omega)}^{1-\theta_k},\qquad\qquad\qquad\qquad\quad
\end{align*}
where $0=(\frac{1}{r_k}-\frac{1}{k+1})\theta_k+\frac{1}{q_k}(1-\theta_k)$ and $\max\{k+1,2\}\leq r_k<+\infty$ and $1\leq q_k<+\infty$ for $k=0,1,...,n-1.$  
It yields that, for $\theta_k>0,$
\begin{align}\label{phint}
\begin{aligned}
\|\ac\phi\|_{L^\infty}\le C\e_0\sum_{k=0}^{n-1}e^{-\bar ct\theta_k}e^{-\bar ct(1-\theta_k)}\left\{(1+t)^\e e^{\bar ct}\right\}^{(\frac{2}{q_k}-1)(1-\theta_k)}\le C \e_0e^{-\bar c\theta_kt}.
\end{aligned}
\end{align}
Then \cref{d2} can be proved by $\norm{\phi}_{L^{\infty}}\leq \|\mr{\phi}\|_{L^{\infty}}+\|\ac{\phi}\|_{L^{\infty}}$.

\end{proof}

\end{document}